\documentclass{amsart}
\usepackage[american]{babel}
\usepackage[T1]{fontenc}
\usepackage[utf8]{inputenc}
\usepackage{amsmath,amsfonts,amstext, dsfont}
\usepackage{amsthm,amssymb,amsmath,amscd,graphicx,epsfig,hhline,mathrsfs,latexsym,url}
\usepackage[neverdecrease]{paralist}
\usepackage{yfonts,dsfont}
\usepackage{tikz}
\usetikzlibrary{shapes.geometric, arrows}
\newcommand{\mc}{\mathcal}
\newcommand{\ms}{\mathscr}

\newcommand{\mb}{\mathbb}
\newcommand{\mr}{\mathrm}
\newcommand{\aveN}{\frac{1}{N}\sum_{n=1}^N}

\newtheorem{Thm}{Theorem}[section] 
\newtheorem*{Rem}{Remark}

\newtheorem{Exam}{Example}
\newtheorem{Le}[Thm]{Lemma}
\newtheorem{Cor}[Thm]{Corollary} 
\newtheorem{Prop}[Thm]{Proposition}
\theoremstyle{definition}
\newtheorem{Def}{Definition}[section]
\theoremstyle{plain}

\newcommand{\lin}{\text{lin}}

\title{Almost everywhere convergence of entangled ergodic averages}

\author{D\'avid Kunszenti-Kov\'acs}
\address{MTA Alfréd Rényi Institute of Mathematics, P.O. Box 127, H-1364 Budapest, Hungary}
\email{daku@renyi.hu}

\thanks{The author has received funding from the European Research Council under the European Union's Seventh Framework Programme (FP7/2007-2013) / ERC grant agreement $\mr{n}^\circ$617747, and from the MTA R\'enyi Institute Lend\"ulet Limits of Structures Research Group.}

\begin{document}

\maketitle

\begin{abstract}
We study pointwise convergence of entangled averages of the form
\[
\frac{1}{N^k}\sum_{1\leq n_1,\ldots, n_k\leq N} T_m^{n_{\alpha(m)}}A_{m-1}T^{n_{\alpha(m-1)}}_{m-1}\ldots A_2T_2^{n_{\alpha(2)}}A_1T_1^{n_{\alpha(1)}} f,
\]
where $f\in L^2(X,\mu)$, $\alpha:\left\{1,\ldots,m\right\}\to\left\{1,\ldots,k\right\}$, and the $T_i$ are ergodic measure preserving transformations on the standard probability space $(X,\mu)$. We show that under some joint boundedness and twisted compactness conditions on the pairs $(A_i,T_i)$, almost everywhere convergence holds for all $f\in L^2$. We also present results for the general $L^p$ case ($1\leq p<\infty$) and for polynomial powers, in addition to continuous versions concerning ergodic flows.
\end{abstract}


\section{Introduction}

Entangled ergodic averages go back to a paper by Accardi, Hashimoto and Obata \cite{AHO}, where these were introduced motivated by questions from quantum stochastics. Given a map $\alpha:\left\{1,\ldots,m\right\}\to\left\{1,\ldots,k\right\}$ and operators $A_i$ ($1\leq i\leq m-1$) and $T_i$ ($1\leq i\leq m$) on a Banach space $E$, the corresponding entangled averages are the multi-Cesàro means
\[
\frac{1}{N^k}\sum_{1\leq n_1,\ldots, n_k\leq N} T_m^{n_{\alpha(m)}}A_{m-1}T^{n_{\alpha(m-1)}}_{m-1}\ldots A_2T_2^{n_{\alpha(2)}}A_1T_1^{n_{\alpha(1)}}.
\]
Originally, the question was about norm convergence of such averages, and this was further studied by Liebscher \cite{liebscher:1999}, Fidaleo \cite{fidaleo:2007,fidaleo:2010,fidaleo:2014} and Eisner and the author \cite{EKK}.\\
In Eisner, K.-K. \cite{EKK2}, attention was turned to pointwise almost everywhere convergence in the context of the $T_i$'s being Koopman operators on function spaces $E=L^p(X,\mu)$ ($1\leq p<\infty$), where $(X,\mu)$ is a standard probability space (i.e. a compact metrizable space with a Borel probability measure). The paper covered the one-parameter ($k=1$) case, and in the present paper, we push the ideas and methods presented there to study pointwise almost everywhere convergence of entangled ergodic averages in their full generality, allowing for multi-parameter entanglement.\\

Note that in what follows, given a measure preserving transformation $S$ on a standard probability space $(X,\mu)$ and a measurable function $g$ on $X$, we shall write $Sg$ for the function $g\circ S$, i.e., we do not make a distinction between the transformation $S$ and the induced contractive operator on $L^q(X)$ ($1\leq q\leq\infty$).\\
Our main result is as follows.

\begin{Thm}\label{thm:main}
Let $m>1$ and $k$ be positive integers, $\alpha:\left\{1,\ldots,m\right\}\to\left\{1,\ldots,k\right\}$ a not necessarily surjective map, and $T_1,T_2,\ldots, T_m$ 
ergodic measure preserving transformations on a standard probability space $(X,\mu)$.
Let $p\in[1,\infty)$, $E:=L^p(X,\mu)$ and let $E=E_{j,r}\oplus E_{j,s}$ be the Jacobs-Glicksberg-deLeeuw decomposition corresponding to $T_j$ $(1\leq j\leq m)$. Let further $A_j\in\mc{L}(E)$ $(1\leq j< m)$ be bounded operators. For a function $f\in E$ and an index $1\leq j\leq m-1$, write $\ms{A}_{j,f}:=\left\{A_jT_j^nf\left|\right.n\in\mb{N}\right\}$. Suppose that the following conditions hold:
\begin{itemize}
\item[(A1)](Twisted compactness)
For any function $f\in E$, index $1\leq j\leq m-1$ and $\varepsilon>0$, there exists a decomposition $E=\mc{U}\oplus \mc{R}$ with $\dim \mc{U}<\infty$ 
such that
\[
P_\mc{R}\ms{A}_{j,f}
\subset B_\varepsilon(0,L^\infty(X,\mu)),
\]
with $P_\mc{R}$ denoting the projection along $\mc{U}$ onto $\mc{R}$.
\item[(A2)](Joint $\mc{L}^\infty$-boundedness)
There exists a constant $C>0$ such that we have
\[\{A_jT^n_j|n\in\mb{N},1\leq j\leq m-1\}\subset B_C(0,\mc{L}(L^\infty(X,\mu)).
\]
\end{itemize}
Then we have the following:
\begin{enumerate}
\item for each $f\in E_{1,s}$, 
\[
\frac{1}{N^k}\sum_{1\leq n_1,\ldots, n_k\leq N} \left|T_m^{n_{\alpha(m)}}A_{m-1}T^{n_{\alpha(m-1)}}_{m-1}\ldots A_2T_2^{n_{\alpha(2)}}A_1T_1^{n_{\alpha(1)}} f \right|\rightarrow 0
\]
pointwise a.e.;
\item if $p=2$, then for each $f\in E_{1,r}$, 
\[
\frac{1}{N^k}\sum_{1\leq n_1,\ldots, n_k\leq N} T_m^{n_{\alpha(m)}}A_{m-1}T^{n_{\alpha(m-1)}}_{m-1}\ldots A_2T_2^{n_{\alpha(2)}}A_1T_1^{n_{\alpha(1)}} f
\]
converges pointwise a.e..
\end{enumerate}
\end{Thm}

\begin{Rem}
Note that it was proven in \cite{EKK2} that the Volterra operator $V$ on $L^2([0,1])$ defined through
\[
(Vf)(x):=\int_{0}^x f(z) \mr{d}z
\]
as well as all of its powers can be decomposed into a finite sum of operators, each of which satisfy conditions (A1) and (A2) when paired with any Koopman operator. Hence the conclusions of Theorem \ref{thm:main} apply whenever the operators $A_i$ are chosen to be powers of $V$.
\end{Rem}


\section{Notations and tools}\label{sec:prelim}

Before proceeding to the proof of our main result, we need to clarify some of the notions used, and introduce notations that will simplify our arguments.

Let $\mc{N}$ denote the set of all bounded sequences $\{a_n\}\subset \ell^\infty(\mb{C})$ satisfying
$$
\lim_{N\to\infty}\aveN |a_n|=0.
$$
By the Koopman-von Neumann lemma, see e.g.~Petersen \cite[p. 65]{petersen:1983}, $(a_n)\in\mc{N}$ if and only if it lies in $\ell^\infty$ and converges to $0$ along a sequence of density $1$.\\

\begin{Def}
Given a Banach space $E$ and an operator $T\in\mc{L}(E)$, the operator $T$ is said to have \emph{relatively weakly compact orbits} if for each $f\in E$ the orbit set $\left\{ T^nf|n\in\mb{N}^+\right\}$ is relatively weakly closed in $E$. For any such operator, there exists a corresponding \emph{Jacobs-Glicksberg-deLeeuw} decomposition of the form (cf. \cite[Theorem II.4.8]{eisner-book})
$$E=E_r \oplus E_s,$$
where
\begin{eqnarray*}
E_r&:=&\overline\lin\{f\in E:\ Tf=\lambda f\mbox{ for some } |\lambda|=1\},\\
E_s&:=&\{f\in E:\ (\varphi( T^nf))\in\mc{N} \mbox{ for every } \varphi\in E'\}.
\end{eqnarray*}
\end{Def}

Recall that power bounded operators on reflexive Banach spaces as well as positive contractions $T$ on $L^1(X,\mu)$ with $T\mathds{1}=\mathds{1}$ (cf. \cite[Theorem II.4.8]{eisner-book}) all have relatively weakly compact orbits. In particular, for all $1\leq p<\infty$, any Koopman operator on $E=L^p(X,\mu)$ has this property, and thus the corresponding decomposition above exists.

\begin{Def}
A sequence $(a_n)_{n=1}^\infty\subset \mb{C}$ is called a \emph{good weight for the pointwise ergodic theorem} if for every measure preserving system $(X,\mu,T)$ and every $f\in L^1(X,\mu)$, the weighted ergodic averages 
$$
\aveN a_n T^nf
$$
converge almost everywhere as $N\to\infty$.
\end{Def}

Denote by $\ms{P}\subset \ell^\infty$ the set of almost periodic sequences, i.e., uniform limits of finite linear combinations of sequences of the form $(\lambda^n)$, $|\lambda|=1$. 
Such sequences play an important role in pointwise ergodic theorems.
Indeed, every element in  $\ms{P}$ is a good weight for the pointwise ergodic theorem.  Also, it can easily be checked that the set $\ms{P}$ is closed under (elementwise) multiplication.
For more details and the first part of the following example we refer to e.g.~\cite{E}.

\begin{Exam}\label{ex:alm-per}
\begin{enumerate}
\item
Let $T$ have relatively weakly compact orbits on a Banach space $E$, $f\in E_r$ and $\varphi\in E'$. Then $(\varphi(T^n f))\in \ms{P}$.
\item
Let $(a_n)_{n\in\mb{N}^+}\subset\ell^\infty$ 
such that for some $(q_n)_{n\in\mb{N}^+}\in\ell^1$ and $(\gamma_n)\subset \mb{C}$ with  
$|\gamma_n|=1$, $n\in\mb{N}$,
$$
a_n=
\sum_{k=0}^\infty \gamma_k^n\cdot q_k\quad \forall n\in\mb{N}.
$$
Then $(a_n)\in \ms{P}$.
\end{enumerate}
\end{Exam}

%
%
%
%
%

%
%

\section{Proof of Theorem \ref{thm:main}}\label{sec:general-case}

We shall proceed by successive splitting and reduction. For each operator $T_i$, starting from $T_2$, we split the functions it is applied to into several terms using condition (A1). Most of the obtained terms can be easily dealt with, but for the remaining ''difficult'' terms, we move on to $T_{i+1}$, up to and including $T_m$.\\
We first prove part (1), and then use this result to complete the proof for part (2). The details for the cases $m=2$ and $m=3$ will be worked out fully, and we shall then show how the proof extends to $m>3$.

Let $f\in E$ and $\epsilon>0$ be given. Then by assumption (A1) we have a finite-dimensional subspace $\mc{U}=\mc{U}(f,\varepsilon/C^{m-1})\subset E$ and a decomposition $E=\mc{U}\oplus \mc{R}$ such that
\[
P_\mc{R}\ms{A}_{1,f}\subset B_{\varepsilon/C^{m-1}}(0,L^\infty(X,\mu)).
\]
Choose a maximal linearly independent set $g_1,\ldots,g_\ell$ in $\mc{U}$. We then for each $n\in\mb{N}^+$ have
\[
A_1T_1^{n}f=\lambda_{1,n}g_1+\ldots+\lambda_{\ell,n}g_\ell+r_n
\]
for appropriate coefficients $\lambda_{j,n}\in\mb{C}$ and some remainder term $r_n\in \mc{R}$ with $\|r_n\|_\infty<\varepsilon/C^{m-2}$. There exist linear forms $\varphi_1,\ldots, \varphi_\ell\in E'$ such that 
$$
\varphi_j(g_i)=\delta_{i,j}\quad \mbox{and}\quad \varphi_j|_{\mc{R}}=0\quad  \mbox{ for every } i,j\in\{1,\ldots,\ell\}.
$$
We then have
\[
\lambda_{j,n}= \varphi_j( A_1T_1^nf)= (A_1^*\varphi_j)(T_1^nf),
\]
therefore 
$$
\left|\lambda_{j,n}\right|\leq \|f\| \cdot\|A_1^*\| \max_{j\in\{1,\ldots,\ell\}}\| \varphi_j\|=:c
$$ 
and, if $f\in E_{1,s}$,  $(\lambda_{j,n})_{n\in\mb{N}^+}\in\ms{N}$. Note that $c$ depends on $\varepsilon$.

With this splitting we have that
\begin{eqnarray*}
&&\frac{1}{N^k}\sum_{1\leq n_1,\ldots, n_k\leq N} T_m^{n_{\alpha(m)}}A_{m-1}T^{n_{\alpha(m-1)}}_{m-1}\ldots A_2T_2^{n_{\alpha(2)}}A_1T_1^{n_{\alpha(1)}} f\\
&=&\frac{1}{N^k}\sum_{1\leq n_1,\ldots, n_k\leq N} T_m^{n_{\alpha(m)}}A_{m-1}T^{n_{\alpha(m-1)}}_{m-1}\ldots A_2T_2^{n_{\alpha(2)}} r_{n_{\alpha(1)}}\\
&&+\sum_{j=1}^\ell \frac{1}{N^k}\sum_{1\leq n_1,\ldots, n_k\leq N} T_m^{n_{\alpha(m)}}A_{m-1}T^{n_{\alpha(m-1)}}_{m-1}\ldots A_2T_2^{n_{\alpha(2)}} \lambda_{j,n_{\alpha(1)}}g_j,
\end{eqnarray*}
and we shall investigate the multi-Cesàro convergence of each term separately.

Since $r_{n_{\alpha(1)}}\in L^\infty$, using condition (A2) and the fact that $T_m$ is an $L^\infty$-isometry we have for the first term the inequality
\begin{eqnarray*}
&& \frac{1}{N^k}\sum_{1\leq n_1,\ldots, n_k\leq N} \left|T_m^{n_{\alpha(m)}}A_{m-1}T^{n_{\alpha(m-1)}}_{m-1}\ldots A_2T_2^{n_{\alpha(2)}} r_{n_{\alpha(1)}}\right|(x)\\
&\leq& C^{m-2}
\frac{1}{N}\sum_{n_{\alpha(1)=1}}^N
\|r_{n_{\alpha(1)}}\|_\infty<\varepsilon
\end{eqnarray*}
for all $x$ from a set $R\subset X$ with $\mu(R)=1$.\\
It remains to show that the second term also converges in the required sense.

\noindent We first turn our attention to part (1), and assume that $f\in E_{1,s}$.
Recall that we then have for each $1\leq j\leq \ell$ that $(\lambda_{j,n})_{n\in\mb{N}^+}\in\ms{N}$.
Let us fix $1\leq j\leq \ell$, and consider the term 
\[
\frac{1}{N^k}\sum_{1\leq n_1,\ldots, n_k\leq N} T_m^{n_{\alpha(m)}}A_{m-1}T^{n_{\alpha(m-1)}}_{m-1}\ldots A_2T_2^{n_{\alpha(2)}} \lambda_{j,n_{\alpha(1)}}g_j.
\]

Assume first that $m=2$. In this case we choose a function $\widetilde{g}_j\in L^\infty$ such that $\|g_j-\widetilde{g}_j\|_1\leq\|g_j-\widetilde{g}_j\|_p<\varepsilon/c\ell$.
Then
\begin{eqnarray*}
&&
\frac{1}{N^k}\sum_{1\leq n_1,\ldots, n_k\leq N} |T_2^{n_{\alpha(2)}}\lambda_{j,n_{\alpha(1)}} g_j|\\
&\leq&
\frac{1}{N^k}\sum_{1\leq n_1,\ldots, n_k\leq N} |T_2^{n_{\alpha(2)}}\lambda_{j,n_{\alpha(1)}} (g_j-\widetilde{g}_j)|
+
\frac{1}{N^k}\sum_{1\leq n_1,\ldots, n_k\leq N} |T_2^{n_{\alpha(2)}}\lambda_{j,n_{\alpha(1)}} \widetilde{g}_j|.
\end{eqnarray*}

Since $(\lambda_{j,n})_{n\in\mb{N}^+}\in\ms{N}$, the second term satisfies by (A2)
\[
\frac{1}{N^k}\sum_{1\leq n_1,\ldots, n_k\leq N} |T_2^{n_{\alpha(2)}}\lambda_{j,n_{\alpha(1)}} \widetilde{g}_j|\leq 	\|\widetilde{g}_j\|_\infty\cdot\frac{1}{N}\sum_{n_{\alpha(1)}=1}^N |\lambda_{j,n_{\alpha(1)}}|\to 0
\]
for all $x$ from a set $P_j\subset R$ with $\mu(P_j)=1$. Set $Q:=\cap_{j=1}^\ell P_j$.

It now remains to treat the first term, 
\[
\frac{1}{N^k}\sum_{1\leq n_1,\ldots, n_k\leq N} |T_2^{n_{\alpha(2)}}\lambda_{j,n_{\alpha(1)}} (g_j-\widetilde{g}_j)|.
\]
Applying Birkhoff's pointwise ergodic theorem to the operator $T_2$ and the function $|(g_j-\widetilde{g}_j)|$, there exists a set $S_{j,\varepsilon}\subset Q$ with $\mu(S_{j,\varepsilon})=1$ such that
\[
\lim_{N\to\infty}\frac{1}{N} \sum_{n_{\alpha(2)}=1}^N T_2^{n_{\alpha(2)}} |(g_j-\widetilde{g}_j)|=\|(g_j-\widetilde{g}_j)\|_1.
\]
Since the sequence $(\lambda_{j,n})_{n\in\mb{N}^+}$ is bounded in absolute value by $c$, we thus obtain
\begin{eqnarray*}
&&\overline{\lim_{N\to\infty}}
\frac{1}{N^k}\sum_{1\leq n_1,\ldots, n_k\leq N} |T_2^{n_{\alpha(2)}}\lambda_{j,n_{\alpha(1)}} (g_j-\widetilde{g}_j)|(x)\\
&\leq&
c\cdot\overline{\lim_{N\to\infty}}
\frac{1}{N^k}\sum_{1\leq n_1,\ldots, n_k\leq N} |T_2^{n_{\alpha(2)}}(g_j-\widetilde{g}_j)|(x)
=c\|(g_j-\widetilde{g}_j)\|_1<\varepsilon/\ell
\end{eqnarray*}
for each $x\in S_{j,\varepsilon}$.

Summing over all $1\leq j\leq \ell$, we have for each $x\in\cap_{j=1}^\ell S_{j,\varepsilon}=:S_\varepsilon$ that

\begin{eqnarray*}
&&
\overline{\lim_{N\to\infty}}
\frac{1}{N^k}\sum_{1\leq n_1,\ldots, n_k\leq N} |T_2^{n_{\alpha(2)}}A_1T_1^{n_{\alpha(1)}} f|\\
&\leq&\overline{\lim_{N\to\infty}}\frac{1}{N^k}\sum_{1\leq n_1,\ldots, n_k\leq N} |A_2T_2^{n_{\alpha(2)}} r_{n_{\alpha(1)}}|(x)\\
&&+\sum_{j=1}^\ell \overline{\lim_{N\to\infty}}
\frac{1}{N^k}\sum_{1\leq n_1,\ldots, n_k\leq N} |T_2^{n_{\alpha(2)}} \lambda_{j,n_{\alpha(1)}}\widetilde{g}_j|(x)\\
&&+\sum_{j=1}^\ell \overline{\lim_{N\to\infty}}
\frac{1}{N^k}\sum_{1\leq n_1,\ldots, n_k\leq N} |T_2^{n_{\alpha(2)}} \lambda_{j,n_{\alpha(1)}}(g_j-\widetilde{g}_j)|(x)
\\
&<&\varepsilon+0+\ell\cdot\varepsilon/\ell=2\varepsilon.
\end{eqnarray*}

But $\mu(S_\varepsilon)=1$, and hence we are done (with the case m=2 of part (1)).

Now assume $m>2$ and fix $1\leq j\leq \ell$.
Again by assumption (A1), there exists a further finite dimensional subspace $\mc{U}_j=\mc{U}(g_j,\varepsilon/\ell C^{m-3})\subset E$ and a corresponding decomposition $E=\mc{U}_j\oplus\mc{R}_j$ such that $P_{\mc{R}_j}\ms{A}_{2,g_j}\subset B_{\varepsilon/\ell C^{m-3}}(0,L^\infty(X,\mu))$. Choose a maximal linearly independent set $g_{1,j},g_{2,j},\ldots,g_{\ell_j,j}$ in $\mc{U}_j$ and let\\ $\varphi_{1,j},\ldots, \varphi_{\ell_j,j}\in E'$ have the property
$$
\varphi_{i,j}(g_{l,j})=\delta_{i,l}\quad \mbox{and}\quad \varphi_{i,j}|_{\mc{R}_j}=0\quad  \mbox{ for every } i,l\in\{1,\ldots,\ell_j\}.
$$
Then we may for each $n\in\mb{N}^+$ write
\[
A_2T_2^n g_j=\lambda_{1,j,n}g_{1,j}+\ldots+\lambda_{\ell_j,j,n}g_{\ell_j,j}+r_{j,n}
\]
for appropriate coefficients $\lambda_{i,j,n}\in\mb{C}$ ($1\leq i\leq \ell_j$) and remainder term $r_{j,n}\in \mc{R}_j$ with $\|r_{j,n}\|_\infty<\varepsilon/\ell C^{m-3}$, and we have
\[
\lambda_{i,j,n}= \varphi_{i,j}( A_2T_2^n g_j)=(A_2^*\varphi_{i,j})(T_2^n g_j).
\]
It follows that 
$$
\left|\lambda_{i,j,n}\right|\leq \|g_j\|\cdot\|A_2^*\| \cdot\max_{i\in\{1,\ldots,\ell_j\}}\|\varphi_{i,j}\|=:c_j.
$$

Since $\|r_{j,n_{\alpha(2)}}\|_\infty<\varepsilon/\ell C^{m-3}$, using condition (A2) and the fact that $T_m$ is an $L^\infty$-isometry we have for the contribution of the $r_{j,n}$ terms that
\begin{eqnarray*}
&& \frac{1}{N^k}\sum_{1\leq n_1,\ldots, n_k\leq N} \left|T_m^{n_{\alpha(m)}}A_{m-1}T^{n_{\alpha(m-1)}}_{m-1}\ldots A_3T_3^{n_{\alpha(3)}} \lambda_{j,n_{\alpha(1)}} r_{j,n_{\alpha(2)}}\right|(x)\\
&\leq& C^{m-3}
\frac{1}{N^k}\sum_{1\leq n_1,\ldots, n_k\leq N}
\|\lambda_{j,n_{\alpha(1)}} r_{n_{\alpha(2)}}\|_\infty<\frac{\varepsilon}{\ell} \sum_{n=1}^N |\lambda_{j,n}|
\end{eqnarray*}
for all $x$ from some set $R_j\subset R$ with $\mu(R_j)=1$. But since $(\lambda_{j,n})_{n\in\mb{N}^+}\in\mc{N}$, the Cesàro limit of $(|\lambda_{j,n}|)_{n\in\mb{N}}$ is equal to zero, so we in fact have
\[
\lim_{N\to\infty}
 \frac{1}{N^k}\sum_{1\leq n_1,\ldots, n_k\leq N} \left|T_m^{n_{\alpha(m)}}A_{m-1}T^{n_{\alpha(m-1)}}_{m-1}\ldots A_3T_3^{n_{\alpha(3)}} \lambda_{j,n_{\alpha(1)}} r_{j,n_{\alpha(2)}}\right|(x)
=0
\]
for all $x\in R_j$.

We now turn our attention to the term
\[
\frac{1}{N^k}\sum_{1\leq n_1,\ldots, n_k\leq N} \left| T_m^{n_{\alpha(m)}}A_{m-1}T^{n_{\alpha(m-1)}}_{m-1}\ldots A_3T_3^{n_{\alpha(3)}} \lambda_{j,n_{\alpha(1)}} \lambda_{i,j,n_{\alpha(2)}} g_{i,j}\right|.
\]
If $m=3$, then as above for the case $m=2$, we split each $g_{i,j}$ ($1\leq i\leq \ell_j$) into $\widetilde{g}_{i,j}\in L^\infty$ and the remainder $g_{i,j}-\widetilde{g}_{i,j}$ such that
$\|g_{i,j}-\widetilde{g}_{i,j}\|_1\leq \|g_{i,j}-\widetilde{g}_{i,j}\|_p\leq\varepsilon/cc_j\ell\ell_j$.
Then
\begin{eqnarray*}
&& \frac{1}{N^k}\sum_{1\leq n_1,\ldots, n_k\leq N} \left|T_3^{n_{\alpha(3)}} \lambda_{j,n_{\alpha(1)}} \lambda_{i,j,n_{\alpha(2)}} g_{i,j}\right|(x)\\
&\leq&
 \frac{1}{N^k}\sum_{1\leq n_1,\ldots, n_k\leq N} \left|T_3^{n_{\alpha(3)}} \lambda_{j,n_{\alpha(1)}} \lambda_{i,j,n_{\alpha(2)}} \widetilde{g}_{i,j}\right|(x)\\
&&+ \frac{1}{N^k}\sum_{1\leq n_1,\ldots, n_k\leq N} \left|T_3^{n_{\alpha(3)}} \lambda_{j,n_{\alpha(1)}} \lambda_{i,j,n_{\alpha(2)}} (g_{i,j}-\widetilde{g}_{i,j})\right|(x).
\end{eqnarray*}

Since $(\lambda_{j,n})_{n\in\mb{N}^+}\in\ms{N}$,  and $(|\lambda_{i,j,n}|)_{n\in\mb{N}^+}$ is bounded by $c_j$, the first term satisfies by (A2)
\[
\frac{1}{N^k}\sum_{1\leq n_1,\ldots, n_k\leq N} \left|T_3^{n_{\alpha(3)}} \lambda_{j,n_{\alpha(1)}} \lambda_{i,j,n_{\alpha(2)}} \widetilde{g}_{i,j}\right|(x)
\leq
c_j\|\widetilde{g}_{i,j}\|_\infty\cdot\frac{1}{N}\sum_{n_{\alpha(1)}=1}^N |\lambda_{j,n_{\alpha(1)}}|\to 0
\]
for all $x$ from a set $P_{i,j}\subset R_j$ with $\mu(P_{i,j})=1$. Set $Q_j:=\cap_{i=1}^{\ell_j} P_{i,j}$.

Applying Birkhoff's pointwise ergodic theorem this time to the operator $T_3$ and the function $|(g_{i,j}-\widetilde{g}_{i,j})|$, there exists a set $S_{i,j,\varepsilon}\subset Q_j$ with $\mu(S_{i,j,\varepsilon})=1$ such that
\[
\lim_{N\to\infty}\frac{1}{N} \sum_{n_{\alpha(3)}=1}^N T_3^{n_{\alpha(3)}} |(g_{i,j}-\widetilde{g}_{i,j})|=\|(g_{i,j}-\widetilde{g}_{i,j})\|_1.
\]
Since the sequence $(\lambda_{j,n})_{n\in\mb{N}^+}$ is bounded in absolute value by $c$ and $(|\lambda_{i,j,n}|)_{n\in\mb{N}^+}$ is bounded by $c_j$, we thus obtain
\begin{eqnarray*}
&&\overline{\lim_{N\to\infty}}
\frac{1}{N^k}\sum_{1\leq n_1,\ldots, n_k\leq N} \left|T_3^{n_{\alpha(3)}} \lambda_{j,n_{\alpha(1)}} \lambda_{i,j,n_{\alpha(2)}} (g_{i,j}-\widetilde{g}_{i,j})\right|(x)\\
&\leq&
cc_j\cdot\overline{\lim_{N\to\infty}}
\frac{1}{N^k}\sum_{1\leq n_1,\ldots, n_k\leq N} |T_3^{n_{\alpha(3)}}(g_{i,j}-\widetilde{g}_{i,j})|(x)
=cc_j\|(g_j-\widetilde{g}_j)\|_1<\varepsilon/\ell\ell_j
\end{eqnarray*}
for each $x\in S_{i,j,\varepsilon}$.

Summing over all $1\leq i\leq \ell_j$ and then $1\leq j\leq\ell$ we therefore obtain
\begin{eqnarray*}
&&
\overline{\lim_{N\to\infty}}
\frac{1}{N^k}\sum_{1\leq n_1,\ldots, n_k\leq N} |T_3^{n_{\alpha(3)}}A_2T_2^{n_{\alpha(2)}}A_1T_1^{n_{\alpha(1)}} f|(x)\\
&\leq&\overline{\lim_{N\to\infty}}\frac{1}{N^k}\sum_{1\leq n_1,\ldots, n_k\leq N} |T_3^{n_{\alpha(3)}}A_2T_2^{n_{\alpha(2)}} r_{n_{\alpha(1)}}|(x)\\
&&+\sum_{j=1}^\ell \overline{\lim_{N\to\infty}}
\frac{1}{N^k}\sum_{1\leq n_1,\ldots, n_k\leq N} |T_3^{n_{\alpha(3)}} \lambda_{j,n_{\alpha(1)}} r_{j,n_{\alpha(2)}}|(x)\\
&&+\sum_{j=1}^\ell \sum_{i=1}^{\ell_j} \overline{\lim_{N\to\infty}}
\frac{1}{N^k}\sum_{1\leq n_1,\ldots, n_k\leq N} |T_3^{n_{\alpha(3)}} \lambda_{j,n_{\alpha(1)}}\lambda_{i,j,n_{\alpha(2)}} \widetilde{g}_{i,j}|(x)\\
&&+\sum_{j=1}^\ell \sum_{i=1}^{\ell_j} \overline{\lim_{N\to\infty}}
\frac{1}{N^k}\sum_{1\leq n_1,\ldots, n_k\leq N} |T_3^{n_{\alpha(3)}} \lambda_{j,n_{\alpha(1)}}\lambda_{i,j,n_{\alpha(2)}}(g_{i,j}-\widetilde{g}_{i,j})|(x)
\\
&<&\varepsilon+0+0+\sum_{j=1}^\ell \ell_j\cdot\varepsilon/\ell\ell_j=2\varepsilon.
\end{eqnarray*}
for all $x\in\cap_{j=1}^\ell\cap_{i=1}^{\ell_j} S_{i,j,\varepsilon}=:S_\varepsilon$, where $\mu(S_\varepsilon)=1$.\\
This concludes the case $m=3$ of part (1).
For general values of $m>3$, we follow the same procedure.\\
Fix $\varepsilon>0$ and $f\in E_{1,s}$, and until and including reaching $A_{m-1}T_{m-1}^{n_{\alpha(m-1)}}$, we split the current functions $A_iT_i^ng_*$ according to property (A1) into a remainder term $r_{*,n}$ and a finite linear combination of functions $g_{b,*}$, with the new coefficients $\lambda_{b,*,n}$ forming a bounded sequence.\\
Since each $(\lambda_{j,n})_{n\in\mb{N}^+}$ lies in $\mc{N}$, all remainder terms beyond the $r_n$'s will have a zero contribution to the pointwise multi-Cesàro means (cf. the contribution of the $r_{j,n}$'s in the case $m=3$). The averaging out of the terms with $r_n$ will yield a contribution to the limes superior of at most $\varepsilon$.\\
Once we have reached the stage $T_m^{n_{\alpha(m)}}g_{b,*}$, we split each $g_{b,*}$ further into a bounded function $\widetilde{g}_{b,*}$, and a term $(g_{b,*}-\widetilde{g}_{b,*})$ with very small $L^1$ norm.\\
Again thanks to $(\lambda_{j,n})_{n\in\mb{N}^+}\in\mc{N}$ and all other $\lambda_{*,n}$ sequences being bounded, the former terms will have a zero contribution, whilst the latter terms will, due to Birkhoff's pointwise ergodic theorem, contribute to the limes superior with a total less than $\varepsilon$. Note that this is possible because in each step the number of functions we split into is only dependent on the current functions to be split and the value of $\varepsilon$ (cf. the details of the case $m=3$).

\vspace{0.1cm}

For part (2), assume $p=2$ and note that 
eigenfunctions in $E_{1,r}$ pertaining to different eigenvalues are always orthogonal. Let so $f\in E_{1,r}$ be fixed, and let $\left\{h_v\right\}_{v\in V}$ be an orthonormal basis in $E_{1,r}$ of eigenvectors pertaining to unimodular eigenvalues $\left\{\beta_v\right\}_{v\in V}$. Note that $V$ is a countable set. Then we can write $f=\sum_{v\in V} d_v h_v$, for some $\ell^2$-sequence $(d_v)$, and we have
\begin{eqnarray*}
\lambda_{j,n}
 =\langle T_1^nf,A_1^*\varphi_j\rangle=\big\langle\sum_{v\in V}\beta_v^nd_v h_v,A_1^*\varphi_j\big\rangle=\sum_{v\in V}\beta_v^n \left(d_v\langle h_v,A_1^*\varphi_j\rangle\right),
\end{eqnarray*}
and so for each $1\leq j\leq \ell$ we have that 
$(\lambda_{j,n})_n\in\ms{P}$ 
since $\left(d_v\langle h_v,A_1^*\varphi_j\rangle\right)\in l^1$ by the Cauchy-Schwarz inequality. 

For each $1\leq j\leq \ell$, we may split $g_j$ into a stable and a reversible part with respect to $T_2$, i.e. $g_j=g_j^s+g_j^r$ with $g_j^s\in E_{2,s}$ and $g_j^r\in E_{2,r}$. Then we have
\begin{eqnarray*}
&&\sum_{j=1}^\ell T_m^{n_{\alpha(m)}}A_{m-1}T^{n_{\alpha(m-1)}}_{m-1}\ldots A_2T_2^{n_{\alpha(2)}} \lambda_{j,n_{\alpha(1)}}g_j\\
&=&\sum_{j=1}^\ell T_m^{n_{\alpha(m)}}A_{m-1}T^{n_{\alpha(m-1)}}_{m-1}\ldots A_2T_2^{n_{\alpha(2)}} \lambda_{j,n_{\alpha(1)}}g_j^r\\
&&+\sum_{j=1}^\ell T_m^{n_{\alpha(m)}}A_{m-1}T^{n_{\alpha(m-1)}}_{m-1}\ldots A_2T_2^{n_{\alpha(2)}} \lambda_{j,n_{\alpha(1)}}g_j^s.
\end{eqnarray*}
To complete the proof of part (2), it remains to be shown that the multi-Cesàro means of the Left Hand Side converges pointwise for almost all $x\in R$. To that end, we shall show that each of the two sums on the Right Hand Side have this property.

Indeed, for the second sum, using part (1) applied to $(m-1)$ pairs $A_iT_i^n$, we obtain that
\begin{eqnarray*}
&&\left|
\frac{1}{N^k}\sum_{1\leq n_1,\ldots, n_k\leq N} \sum_{j=1}^\ell T_m^{n_{\alpha(m)}}A_{m-1}T^{n_{\alpha(m-1)}}_{m-1}\ldots A_2T_2^{n_{\alpha(2)}} \lambda_{j,n_{\alpha(1)}}g_j^s
\right|(x)\\
&\leq&
\frac{1}{N^k}\sum_{1\leq n_1,\ldots, n_k\leq N} \sum_{j=1}^{\ell}
\left|
T_m^{n_{\alpha(m)}}A_{m-1}T^{n_{\alpha(m-1)}}_{m-1}\ldots A_2T_2^{n_{\alpha(2)}} \lambda_{j,n_{\alpha(1)}}g_j^s
\right|(x)\\
&\leq&
\frac{1}{N^k}\sum_{1\leq n_1,\ldots, n_k\leq N} \sum_{j=1}^{\ell}
c \left|
T_m^{n_{\alpha(m)}}A_{m-1}T^{n_{\alpha(m-1)}}_{m-1}\ldots A_2T_2^{n_{\alpha(2)}} g_j^s
\right|(x)\\
&=&
c\cdot \sum_{j=1}^{\ell}
\left(
\frac{1}{N^k}\sum_{1\leq n_1,\ldots, n_k\leq N}
\left|
T_m^{n_{\alpha(m)}}A_{m-1}T^{n_{\alpha(m-1)}}_{m-1}\ldots A_2T_2^{n_{\alpha(2)}} g_j^s
\right|(x)
\right)\to 0
\end{eqnarray*}
for all $x$ in a set $S\subset R$ with $\mu(S)=1$.

Now let us turn our attention to the first sum, involving the reversible parts $g_j^r$.\\
In case $m=2$, this is in fact of the simple form
\begin{equation*}
\frac{1}{N^k}\sum_{1\leq n_1,\ldots, n_k\leq N} \sum_{j=1}^\ell T_2^{n_{\alpha(2)}} \lambda_{j,n_{\alpha(1)}}g_j^r (x)
=\sum_{j=1}^\ell \frac{1}{N^k}\sum_{1\leq n_1,\ldots, n_k\leq N} \lambda_{j,n_{\alpha(1)}} T_2^{n_{\alpha(2)}} g_j^r (x).
\end{equation*}
Recall that the sequences $(\lambda_{j,n})_{n\in\mb{N}^+}$ are almost periodic, hence on the one hand are good weights for the pointwise ergodic theorem, but also converge in the Cesàro sense.
Using the former property in case $\alpha(1)=\alpha(2)$ and the latter otherwise, we therefore obtain that the above expression converges for almost all $x\in S$.
This concludes the case $m=2$ for part (2).\\
Now assume $m\geq 3$, and use property (A1) applied to $A_2,T_2$ to obtain the following.
For each $1\leq j\leq \ell$ there exists a further finite dimensional subspace
\[
\mc{U}_j=\mc{U}(g_j^r,\varepsilon/cC^{m-2})\subset E
\]
and a corresponding decomposition $E=\mc{U}_j\oplus\mc{R}_j$
such that
\[
P_{\mc{R}_j}\ms{A}_{2,g_j^r}\subset B_{\varepsilon/cC^{m-2}}(0,L^\infty(X,\mu)).
\]
 Let $g_{1,j},g_{2,j},\ldots,g_{\ell_j,j}$ be an orthonormal basis in $\mc{U}_j$. Then we may for each $n\in\mb{N}^+$ write
\[
A_2T_2^n g_j^r=\lambda_{1,j,n}g_{1,j}+\ldots+\lambda_{\ell_j,j,n}g_{\ell_j,j}+r_{j,n}
\]
for appropriate coefficients $\lambda_{i,j,n}\in\mb{C}$ ($1\leq i\leq \ell_j$) and remainder term $r_{j,n}\in \mc{R}_j$ with $\|r_{j,n}\|_\infty<\varepsilon/C^{a-2}$, and we have
\[
\lambda_{i,j,n}= \langle A_2T_2^n g_j^r,\varphi_{i,j}\rangle=\langle T_2^n g_j^r,A_2^*\varphi_{i,j}\rangle,
\]
where each $\varphi_{i,j}$ is orthogonal to $\mc{R}_j$ and $\langle g_{l,j},\varphi_{i,j}\rangle=\delta_{l,i}$. 
(Note that in case $\mc{R}_j\perp\mc{U}_j$, we may choose 
$\varphi_{i,j}:=g_{i,j}$.)
Thus 
$$
\left|\lambda_{i,j,n}\right|\leq \|g_j^r\|_2\cdot\|A_2^*\|\max\{\|\varphi_{i,j}\|_2,\, i=1,\ldots,k_j\}=:c_j
$$
and also $(\lambda_{i,j,n})_{n\in\mb{N}^+}\in\ms{P}$ by Example \ref{ex:alm-per}. 

Therefore for each $1\leq j\leq \ell$ we have for almost every $x\in X$
\begin{align*}
&\left|\overline{\lim}_{N\to\infty}
\left(
\frac{1}{N^k}\sum_{1\leq n_1,\ldots, n_k\leq N} T_m^{n_{\alpha(m)}}A_{m-1}T^{n_{\alpha(m-1)}}_{m-1}\ldots A_2T_2^{n_{\alpha(2)}} \lambda_{j,n_{\alpha(1)}} g_j^r
\right)(x)
\right.
\\
&
-
\left.
\underline{\lim}_{N\to\infty}
\left(
\frac{1}{N^k}\sum_{1\leq n_1,\ldots, n_k\leq N} T_m^{n_{\alpha(m)}}A_{m-1}T^{n_{\alpha(m-1)}}_{m-1}\ldots A_2T_2^{n_{\alpha(2)}} \lambda_{j,n_{\alpha(1)}} g_j^r
\right)(x)
\right|\\
&\leq
\left|\overline{\lim}_{N\to\infty}
\left(
\frac{1}{N^k}\sum_{1\leq n_1,\ldots, n_k\leq N} T_m^{n_{\alpha(m)}}A_{m-1}T^{n_{\alpha(m-1)}}_{m-1}\ldots A_3T_3^{n_{\alpha(3)}} \lambda_{j,n_{\alpha(1)}} r_{j,n}
\right)(x)
\right.
\\
&
-
\left.
\underline{\lim}_{N\to\infty}
\left(
\frac{1}{N^k}\sum_{1\leq n_1,\ldots, n_k\leq N} T_m^{n_{\alpha(m)}}A_{m-1}T^{n_{\alpha(m-1)}}_{m-1}\ldots A_3T_3^{n_{\alpha(3)}} \lambda_{j,n_{\alpha(1)}} r_{j,n}
\right)(x)
\right|
\\
&+
\sum_{i=1}^{\ell_j}
\left|\overline{\lim}_{N\to\infty}
\left(
\frac{1}{N^k}\sum_{1\leq n_1,\ldots, n_k\leq N} T_m^{n_{\alpha(m)}}
\ldots A_3T_3^{n_{\alpha(3)}} \lambda_{j,n_{\alpha(1)}}\lambda_{i,j,n_{\alpha(2)}} g_{i,j}
\right)(x)
\right.
\\
&
-
\left.
\underline{\lim}_{N\to\infty}
\left(
\frac{1}{N^k}\sum_{1\leq n_1,\ldots, n_k\leq N} T_m^{n_{\alpha(m)}}
\ldots A_3T_3^{n_{\alpha(3)}} \lambda_{j,n_{\alpha(1)}}\lambda_{i,j,n_{\alpha(2)}} g_{i,j}
\right)(x)
\right|.
\end{align*}

Note that the first term on the Right Hand Side is bounded by
\[
2cC^{m-2}\|r_{j,n}\|_\infty<2\varepsilon
\]
for all $x$ from a set $R_j\subset S$ with $\mu(R_j)=1$, and to complete our proof we need to handle the terms of the sum.\\
Now if $m=3$, the sum simplifies to
\begin{eqnarray*}
\sum_{i=1}^{\ell_j}
&&\left|\overline{\lim}_{N\to\infty}
\left(
\frac{1}{N^k}\sum_{1\leq n_1,\ldots, n_k\leq N} T_3^{n_{\alpha(3)}} \lambda_{j,n_{\alpha(1)}}\lambda_{i,j,n_{\alpha(2)}} g_{i,j}
\right)(x)
\right.\\
&&-
\left.
\underline{\lim}_{N\to\infty}
\left(
\frac{1}{N^k}\sum_{1\leq n_1,\ldots, n_k\leq N} T_3^{n_{\alpha(3)}} \lambda_{j,n_{\alpha(1)}}\lambda_{i,j,n_{\alpha(2)}} g_{i,j}
\right)(x)
\right|
\end{eqnarray*}

The sequences $(\lambda_{i,j,n})_{n\in\mb{N}^+}$ and $(\lambda_{j,n})_{n\in\mb{N}^+}$ are almost periodic, hence on the one hand are good weights for the pointwise ergodic theorem, but also converge in the Cesàro sense, and the same is true for their product. In light of these properties, grouping them in a manner similar to what was done for the case $m=2$, we obtain that the above expression converges to $0$ for almost all $x\in R_j$.
This concludes the case $m=3$ for part (2).

In case $m>3$, the above splitting procedure can be extended in a natural way.
In each step, the functions $A_iT_i^ng_*$ are split using property (A1), the remainder terms $r_{*,n}$ giving rise to a total contribution of at most $2\varepsilon$ to the difference of the limsup and liminf. Given that the number of steps is $m-1$, this is a contribution over all steps of $\leq 2(m-1)\varepsilon$.\\
Then each $g_{j,*}$ is further split into its stable and its reversible part with respect to $T_{i+1}$.\\
The stable parts $g_{j,*}^s$ have a contribution of $0$ by part (1) of this theorem, proven above. The splitting is then repeated for the reversible parts $g_{j,*}^r$ until we reach $T_m$. Once at that stage, we group the coefficients $\lambda_{**,n_{\alpha(i)}}$ according to the value of $\alpha(i)$. Note that each sequence of $\lambda$'s is almost periodic, hence any product is as well. For each $b\neq\alpha(m)$ the corresponding product $\prod_{\alpha(i)=b}\lambda_{**,n_{\alpha(i)}}$ will in Cesàro average converge to some limit. For $b=\alpha(m)$ however, the corresponding product is the coefficient sequence of $T^{n_b}g_*$, and here we use the fact that almost periodic sequences are good weights for the pointwise ergodic theorem to obtain convergence (cf. the cases $m=2,3$).\\
In total, the difference of the limsup and the liminf will not exceed $2(m-1)\varepsilon$ for all $x$ from a set of full measure, concluding the proof.

\begin{Rem}
The pointwise limit is -- if it exists -- clearly the same as the stong limit, and takes the form given in \cite[Thm. 3]{EKK}.
\end{Rem}

%
%
\section{Weakly mixing transformations}

In this section we wish to obtain a polynomial version of Theorem \ref{thm:main}, as well as trying to say something about the reversible part for the general $L^p$ case ($1\leq p<\infty$). To this end, we shall need stronger assumptions on the transformations $T_j$.
Ergodicity will not suffice for our purpose, instead we shall require weak mixing.

\begin{Def}[cf. \cite{furst:book} Def. 4.1]
A measure preserving transformation $T$ on a measure 
space $(X,\mu)$ is weakly mixing if $T \times T$ is an ergodic transformation on
$(X \times X, \mu\times\mu)$. 
\end{Def}

The following lemma shows how this class of transformations becomes relevant for us.

\begin{Le}\label{le:weak_mix}
Let $T$ be a weakly mixing measure preserving transformation on a standard probability space $(X,\mu)$, $1<p<\infty$, and $E=L^p(X,\mu)$. Then the reversible part $E_r$ of the Jacobs-Glicksberg-deLeeuw decomposition $E=E_r\oplus E_s$ corresponding to $T$ is one-dimensional and generated by the constant 1 function~$\mathds{1}$.
\end{Le}
\begin{proof}
The case $p=2$ follows from \cite[Thm. 4.30.]{furst:book}, the general case from the arguments presented before that theorem (\cite{furst:book} pp. 96--97).
\end{proof}
This spectral property of weakly mixing transformations simplifies the treatment of the reversible part. Also, as the next proposition shows, it guarantees that polynomial Cesàro means converge to an $L^\infty$ function.

\begin{Prop}[cf. \cite{Bourgain} Thm. 1]\label{prop:bourgain}
Let $T$ be an ergodic measure preserving transformation on a standard probability space $(X,\mu)$ and q(x) a polynomial with integer coefficients taking positive values on $\mb{N}^+$. Then for any $1<p<\infty$ and $f\in L^p(X,\mu)$ the limit
\[
\frac{1}{N} \sum_{n=1}^N T^{q(n)}f
\]
exists almost surely. If in addition $T$ is weakly mixing, the limit is given by the constant function $\int_X f\mr{d}\mu$.
\end{Prop}

Our last ingredient is the following proposition, which allows us to establish properties of the coefficient sequences $(\lambda_{*,n})$ along polynomial indices. Recall that an operator is almost weakly stable if the stable part of the Jacobs-Glicksberg-deLeeuw decomposition is the whole space.

\begin{Prop}[cf. \cite{KK_AWPS} Thm. 1.1]\label{prop:uaws_pol}
Let $T$ be an almost weakly stable contraction on a Hilbert space $H$. Then $T$ is almost weakly polynomial stable, i.e., for any $h\in H$ and non-constant polynomial $q$ with integer coefficients taking positive values on $\mb{N}^+$, the sequence $\{T^{q(j)}h\}_{j=1}^\infty$ is almost weakly stable.
\end{Prop}

Reformulating in the context needed in this paper we obtain the following.

\begin{Cor}\label{cor:awps}
Let $T$ be a weakly mixing measure preserving transformation on the standard probability space $(X,\mu)$, $q$ a non-constant polynomial with integer coefficients taking positive values on $\mb{N}^+$ and $A$ an arbitrary operator on $L^2(X,\mu)$.
Then for any $g,\varphi\in L^2(X,\mu)$ we have that the sequence $\langle AT^{q(n)}g,\varphi\rangle$ is bounded and lies in $\mc{N}$.
\end{Cor}

With these tools in hand, we can now state and prove almost everywhere pointwise convergence of entangled means on Hilbert spaces.

\begin{Thm}\label{thm:main_poly}
Let $m>1$ and $k$ be positive integers, $\alpha:\left\{1,\ldots,m\right\}\to\left\{1,\ldots,k\right\}$ a not necessarily surjective map, and $T_1,T_2,\ldots, T_m$ weakly mixing measure preserving transformations on a standard probability space $(X,\mu)$.
Let $E:=L^2(X,\mu)$ and let $E=E_{j,r}\oplus E_{j,s}$ be the Jacobs-Glicksberg-deLeeuw decomposition corresponding to $T_j$ $(1\leq j\leq m)$. Let further $A_j\in\mc{L}(E)$ $(1\leq j< m)$ be bounded operators. Suppose that the the conditions (A1) and (A2) of Theorem \ref{thm:main} hold.\\
Further, let $q_1,q_2,\ldots, q_k$ be non-constant polynomials with integer coefficients taking positive values on $\mb{N}^+$.
Then we have the following:
\begin{enumerate}
\item for each $f\in E_{1,s}$, 
\[
\frac{1}{N^k}\sum_{1\leq n_1,\ldots, n_k\leq N} \left|T_m^{q_{\alpha(m)}(n_{\alpha(m)})}
\ldots A_2T_2^{q_{\alpha(2)}(n_{\alpha(2)})}A_1T_1^{q_{\alpha(1)}(n_{\alpha(1)})} f \right|\rightarrow 0
\]
pointwise a.e.;
\item for each $f\in E$, the averages 
\[
\frac{1}{N^k}\sum_{1\leq n_1,\ldots, n_k\leq N} T_m^{q_{\alpha(m)}(n_{\alpha(m)})}
\ldots A_2T_2^{q_{\alpha(2)}(n_{\alpha(2)})}A_1T_1^{q_{\alpha(1)}(n_{\alpha(1)})} \mathds{1}
\]
converge pointwise a.e. to the constant function
\[
\left(\langle f,\mathds{1}\rangle\prod_{i=1}^{m-1}\langle A_i\mathds{1},\mathds{1}\rangle\right)\cdot\mathds{1}.
\]
\end{enumerate}
\end{Thm}

\begin{proof}
We shall follow the proof of Theorem \ref{thm:main}, splitting functions in the exact same way, and then take a closer look at why the convergences still hold when averaging along polynomial subsequences.\\
For part (1), the terms $r_n$ are still uniformly small in $L^\infty$, and the Cesàro averages remain small, regardless of the polynomial subsequences. Also, the coefficient sequences $\lambda_{*}$ are now subsequences of the ones in Theorem \ref{thm:main}, and so still bounded. In addition to this boundedness of coefficientss, the sequences $(\lambda_{j,n})_{n\in\mb{N}}$ being in $\mc{N}$ meant that the remainder terms $r_{*,n}$ arising from all futher decompositions average out to zero, and the same applies to the contribution of the bounded terms $\widetilde{g}_{b,*}$. In this polynomial setting this still holds by Corollary \ref{cor:awps}, and so we only have the terms with $g_{b,*}-\widetilde{g}_{b,*}$ left.\\
For these, make use of the fact that all sequences $\lambda_*$ are bounded, and instead of Birkhoff's pointwise ergodic theorem, apply Bourgain's polynomial version, Proposition \ref{prop:bourgain} in the weakly mixing case, to obtain that these terms also average out to a total contribution of at most $\varepsilon$ to the limsup.\\
For part (2), it suffices to extend the proof of part (2) of Theorem \ref{thm:main}, and then investigate the limit. Here the key is that the reversible spaces $E_{i,r}$ are all one dimensional, spanned by $\mathds{1}$, with eigenvalue 1. Hence all coefficient sequences $\lambda_*$ arising from the reversible parts are constant 1, and as such average out to 1 even along polynomial subsequences. Therfore part (1) can be applied whenever a stable term $g_*^s$ comes into play, the remainder terms $r_*$ have a total contribution of at most $2\varepsilon$ in each splitting step, and we are left with the average of the terms that arise when at each split we choose the reversible part, with all coefficients $\lambda_*$ equal to 1. Pointwise convergence then follows from Proposition \ref{prop:bourgain}.\\
When it comes to the limit, by the previous we have that the limit is determined by the purely reversible parts. Denoting by
\[P_i:=\lim_{N\to\infty}\frac{1}{N}\sum_{n=1}^N T_i^n
\]
the mean ergodic projection onto $E_{i,r}=\mr{Fix}(T_i)$ corresponding to $T_i$, this is equal to
\[
P_mA_{m-1}P_{m-1}\ldots A_1P_1f=\left(\langle f,\mathds{1}\rangle\prod_{i=1}^{m-1}\langle A_i\mathds{1},\mathds{1}\rangle\right)\cdot\mathds{1}.
\]
\end{proof}

In addition, weak mixing allows us to extend the convergence proven in Theorem \ref{thm:main} to the reversible part for all $1\leq p<\infty$.

\begin{Thm}\label{thm:main_Lp}
Let $m>1$ and $k$ be positive integers, $\alpha:\left\{1,\ldots,m\right\}\to\left\{1,\ldots,k\right\}$ a not necessarily surjective map, and $T_1,T_2,\ldots, T_m$ weakly mixing measure preserving transformations on a standard probability space $(X,\mu)$. Let $p\in[1,\infty)$, $E:=L^p(X,\mu)$ and let $E=E_{j,r}\oplus E_{j,s}$ be the Jacobs-Glicksberg-deLeeuw decomposition corresponding to $T_j$ $(1\leq j\leq m)$. Let further $A_j\in\mc{L}(E)$ $(1\leq j< m)$ be bounded operators.
Suppose that the conditions (A1) and (A2) of Theorem \ref{thm:main} hold.\\
Then we have for each $f\in E$ that the averages 
\[
\frac{1}{N^k}\sum_{1\leq n_1,\ldots, n_k\leq N} T_m^{n_{\alpha(m)}}A_{m-1}T^{n_{\alpha(m-1)}}_{m-1}\ldots A_2T_2^{n_{\alpha(2)}}A_1T_1^{n_{\alpha(1)}} f
\]
converge pointwise a.e. to the constant function
\[
\left(\langle f,\mathds{1}\rangle\prod_{i=1}^{m-1}\langle A_i\mathds{1},\mathds{1}\rangle\right)\cdot\mathds{1}.
\]
\end{Thm}

\begin{proof}
We again follow in the steps of the proof of Theorem \ref{thm:main}. Writing $f=\langle f,\mathds{1}\rangle\cdot\mathds{1}+(f-\langle f,\mathds{1}\rangle\cdot\mathds{1})$, we have by Lemma \ref{le:weak_mix} that the second term lies in $E_{1,s}$. By Theorem \ref{thm:main} that part averages out to 0, so we only need to prove convergence for the function $\mathds{1}$. The key to part (2) of Theorem \ref{thm:main} was showing that the coefficients $\lambda_{*,n}$ form an almost periodic sequence through the use of a basis of eigenfunctions in the reversible parts $E_{i,r}$. This would in general not be possible for $L^p$ spaces, as a non-orthogonal decomposition of a function into an infinite sum $\sum_{v\in V} d_vh_v$ would leave us with no good bound on the coefficients $d_v$. However now all our operators $T_i$ are weakly mixing, and hence the decomposition of a function $f\in E_{i,r}$ is the trivial $d\cdot\mathds{1}$ for the appropriate $d\in\mb{C}$, with the corresponding eigenvalue being 1. Consequently all coefficient sequences $\lambda_{*,n}$ arising from rotational parts are constant sequences, and a fortiori almost periodic. Therefore the proof presented for Theorem \ref{thm:main} goes through in this case as well.\\
Concerning the limit function itself, note that the proof actually yields (by part (1) of Theorem \ref{thm:main}) that any stable part arising during the splitting averages to zero, and the remainder terms $r_*$ are uniformly small. Hence, again denoting by $P_i$ the mean ergodic projection onto $E_{i,r}=\mr{Fix}(T_i)$ corresponding to $T_i$, the limit function is easily seen to be
\[
P_mA_{m-1}P_{m-1}\ldots A_1P_1f=\left(\langle f,\mathds{1}\rangle\prod_{i=1}^{m-1}\langle A_i\mathds{1},\mathds{1}\rangle\right)\cdot\mathds{1}.
\]
\end{proof}

%
%
%
%

\section{The continuous case}\label{sec:ex}

In this section, we finally turn our attention to a variant of the above results, where we replace the discrete action of the measure preserving operators with the continuous action of measure preserving flows. In other words, the semigroups $\{T_i^n|n\in\mb{N}^+\}$ are replaced by semigroups $\{T_i(t)|t\in[0,\infty)\}$.\\
Note that both Birkhoff's pointwise ergodic theorem, and the Jacobs-Glicksberg-deLeeuw decomposition possess a corresponding continuous version (for the latter, see e.g.~\cite[Theorem III.5.7]{eisner-book}).  Hence, with an analogous proof, we obtain the following continuous version of Theorem \ref{thm:main}.

\begin{Thm}\label{thm:main-cont}
Let $m>1$ and $k$ be positive integers, $\alpha:\left\{1,\ldots,m\right\}\to\left\{1,\ldots,k\right\}$ a not necessarily surjective map and let 
 $(T_1(t))_{t\geq 0}$,$\ldots $,$(T_m(t))_{t\geq 0}$ be ergodic measure preserving flows on a standard probability space $(X,\mu)$.
Let $p\in[1,\infty)$, $E:=L^p(X,\mu)$ and let $E=E_{j,r}\oplus E_{j,s}$ be the Jacobs-Glicksberg-deLeeuw decomposition corresponding to $T_j(\cdot)$ $(1\leq j\leq m)$. Let further $A_j\in\mc{L}(E)$ $(1\leq j< m-1)$ be bounded operators. For a function $f\in E$ and an index $1\leq j\leq m-1$, write $\ms{A}_{j,f}:=\left\{A_jT_j(t)f\left|\right.t\in[0,\infty)\right\}$. Suppose that the following conditions hold:
\begin{itemize}
\item[(A1c)](Twisted compactness)
For any function $f\in E$, index $1\leq j\leq m-1$ and $\varepsilon>0$, there exists a decomposition $E=\mc{U}\oplus \mc{R}$ with $\dim \mc{U}<\infty$
such that
\[
P_{\mc{R}}\ms{A}_{j,f}
\subset B_\varepsilon(0,L^\infty(X,\mu)),
\]
with $P_{\mc{R}}$ denoting the projection onto $\mc{R}$ along $\mc{U}$.
\item[(A2c)](Joint $\mc{L}^\infty$-boundedness)
There exists a constant $C>0$ such that we have
\[\{A_jT_j(t)|\,t\in[0,\infty),1\leq j\leq m-1\}\subset B_C(0,\mc{L}(L^\infty(X,\mu)).
\]
\end{itemize}
Then we have the following:
\begin{enumerate}
\item for each $f\in E_{0,s}$, 
\[
\lim_{\mc{T}\to\infty}\frac{1}{\mc{T}^k}\int_{\left\{t_1,\ldots, t_k\right\}\in [0,\mc{T}]^k} \left|T_m(t_{\alpha(m)})
\ldots A_2T_2(t_{\alpha(2)})A_1T_1(t_{\alpha(1)}) f \right|\rightarrow 0
\]
pointwise a.e.;
\item if $p=2$, then for each $f\in E_{1,r}$, 
\[
\frac{1}{\mc{T}^k}\int_{\left\{ t_1,\ldots, t_k\right\}\in [0,\mc{T}]^k} T_m(t_{\alpha(m)})A_{m-1}T_{m-1}(t_{\alpha(m-1)})\ldots A_2T_2(t_{\alpha(2)})A_1T_1(t_{\alpha(1)}) f
\]
converges pointwise a.e..
\end{enumerate}
\end{Thm}

Since Lemma \ref{le:weak_mix} also extends to the time-continuous case, we have the following continuous version of Theorem \ref{thm:main_Lp}, with a proof analogous to the original.

\begin{Thm}\label{thm:main-cont_Lp}
Let $m>1$ and $k$ be positive integers, $\alpha:\left\{1,\ldots,m\right\}\to\left\{1,\ldots,k\right\}$ a not necessarily surjective map and let $(T_1(t))_{t\geq 0}$,$\ldots $,$(T_m(t))_{t\geq 0}$ be 
weakly mixing measure preserving flows on a standard probability space $(X,\mu)$.
Let $p\in[1,\infty)$, $E:=L^p(X,\mu)$ and let $E=E_{j,r}\oplus E_{j,s}$ be the Jacobs-Glicksberg-deLeeuw decomposition corresponding to $T_j(\cdot)$ $(1\leq j\leq m)$. Let further $A_j\in\mc{L}(E)$ $(1\leq j< m-1)$ be bounded operators. 
Suppose that the conditions (A1c) and (A2c) from Theorem \ref{thm:main-cont_Lp} hold.

Then for each $f\in E$, the multi-Cesàro averages
\[
\frac{1}{\mc{T}^k}\int_{\left\{ t_1,\ldots, t_k\right\}\in [0,\mc{T}]^k} T_m(t_{\alpha(m)})A_{m-1}T_{m-1}(t_{\alpha(m-1)})\ldots A_2T_2(t_{\alpha(2)})A_1T_1(t_{\alpha(1)}) f
\]
converge pointwise a.e. to the constant function
\[
\left(\langle f,\mathds{1}\rangle\prod_{i=1}^{m-1}\langle A_i\mathds{1},\mathds{1}\rangle\right)\cdot\mathds{1}.
\]
\end{Thm}

\end{document}